\documentclass[12pt,a4paper,reqno]{amsart}
\usepackage{amsmath}
\usepackage{amsfonts}
\usepackage{amssymb}
\usepackage{amsthm}
\usepackage{enumerate}
\usepackage{centernot}
\usepackage{fullpage}
\usepackage{mathrsfs}
\usepackage{graphicx}
\usepackage[colorlinks,
linkcolor=blue,
anchorcolor=green,
citecolor=blue
]{hyperref}
\usepackage{color}

\theoremstyle{plain}

  \newtheorem{proposition}[]{Proposition}
  \newtheorem{lemma}[]{Lemma}
  \newtheorem{theorem}[]{Theorem}

  \newtheorem{remark}[]{Remark}

\title[Coexistence of infinite clusters]{Coexistence of infinite clusters for percolation and Ising model on $\mathbb{Z}^d$}
\author{Jianping Jiang}
\address{Yau Mathematical Sciences Center, Tsinghua University, Beijing 100084, China.}
\email{jianpingjiang@tsinghua.edu.cn}
\author{Sike Lang}
\address{Qiuzhen College, Tsinghua University, Beijing 100084, China.}
\email{langsk24@mails.tsinghua.edu.cn}

\begin{document}
\begin{abstract}
For independent bond percolation on $\mathbb{Z}^d$ with parameter $p$, let $p_c^b(d)$ be the critical probability. We prove that for each $d \geq 9$,  there is $\epsilon_d>0$ such that for each $p \in (p_c^b(d), p_c^b(d)+\epsilon_d)$, the complement of the infinite open cluster stochastically dominates a supercritical site percolation on $\mathbb{Z}^d$. This improves the previous results by Grimmett, Holroyd and Kozma 2014, and Bock, Damron, Newman and Sidoravicius 2020. Numerical estimates of $p_c^b(d)$ and $p_c^s(d)$ (the site critical probability) suggest that a similar stochastic domination result holds for all $d \geq 4$.

For the Ising model on $\mathbb{Z}^d$ with inverse temperature $\beta$, let $\beta_c(d)$ be the critical inverse temperature. We prove that for each $d \geq 8$, there is $\epsilon_d>0$  such that for each $\beta \in [0,\beta_c(d)+\epsilon_d)$, the $+$ spins under the minus phase stochastically dominate a supercritical site percolation on $\mathbb{Z}^d$. This improves the previous result of Aizenman, Bricmont and Lebowitz 1987. Numerical estimates of $\beta_c(d)$ and $p_c^s(d)$ suggest that a similar stochastic domination result holds for all $d \geq 5$.
\end{abstract}

\maketitle

\section{Introduction and main results}
In bond (site, respectively) percolation on $\mathbb{Z}^d$, each nearest-neighbor edge (site, respectively) is declared to be open with probability $p$ and closed with probability $1-p$ independently. Let $\mathbb{P}^b_p$ ($\mathbb{P}^s_p$, respectively) be the resulting product measure on $\{0,1\}^{E(\mathbb{Z}^d)}$ ($\{0,1\}^{V(\mathbb{Z}^d)}$, respectively) where $E(\mathbb{Z}^d)$ is the set of all nearest-neighbor edges and $V(\mathbb{Z}^d)$ is the set of vertices on $\mathbb{Z}^d$. The critical probabilities $p_c^b(d)$ and $p_c^s(d)$ are defined by
\[p_c^b(d):=\inf \{p\in [0,1]: \mathbb{P}^b_p(\text{ there is an infinite open cluster})>0\},\]
\[p_c^s(d):=\inf \{p\in [0,1]: \mathbb{P}^s_p(\text{ there is an infinite open cluster})>0\}.\]
It is well known that $p_c^b(d), p_c^s(d) \in (0,1)$ for all $d\geq 2$ (see, e.g., Theorem 1.33 of \cite{Gri99}).

When $p>p_c^b(d)$, it is known that there is a unique infinite open cluster \cite{AKN87,BK89}. In \cite{GHK14}, the following question was asked: is there an infinite connected component after removing all vertices in the infinite open cluster? More precisely, let $\mathcal{X}_p(d)$ be the subgraph of $\mathbb{Z}^d$ obtained after removing all vertices in the infinite open cluster; one can define
\[p_{\textnormal{fin}}^b(d):=\sup\{p\in[0,1]: \mathbb{P}^b_p(\mathcal{X}_p(d) \text{ has an infinite connected component})>0\}.\]

We have $p_{\textnormal{fin}}^b(2)=p_c^b(2)=1/2$ \cite{Kes80} since for $p>p_c^b(2)$, the infinite open cluster contains circuits around each vertex. In \cite{GHK14}, it was proved that $p_{\textnormal{fin}}^b(d) > p_c^b(d)$ for all $d \geq 19$ using the one-arm probability established in \cite{KN11}. Combining with the later result \cite{FH17}, one has $p_{\textnormal{fin}}^b(d) > p_c^b(d)$ for all $d \geq 11$. Using shielded percolation, Bock, Damron, Newman and Sidoravicius improved this result to $d\geq 10$. In this paper, we prove
\begin{theorem}\label{thm:per}
	For each $d \geq 9$, there is $\epsilon_d>0$ such that for each $p \in (p_c^b(d), p_c^b(d)+\epsilon_d)$, $\mathcal{X}_p(d)$ stochastically dominates a site percolation on $\mathbb{Z}^d$ with parameter $\tilde{p}=\tilde{p}(p)>p^s_c(d)$. In particular, we have $p_{\textnormal{fin}}^b(d)>p_c^b(d)$.
\end{theorem}
 
 \begin{remark}
 	Theorem \ref{thm:per} implies the coexistence of a unique infinite open cluster and a unique infinite connected component of $\mathcal{X}_p(d)$ whenever $p \in (p_c^b(d), p_c^b(d)+\epsilon_d)$ and $d\geq 9$; the uniqueness of the latter follows from Theorem 4.3 of \cite{GHK14}.
 \end{remark}
 \begin{remark}
 	As in \cite{GHK14}, one may define plaquette percolation on $\mathbb{Z}^d$ by declaring each plaquette (which is a unit $(d-1)$-cube) to be open with probability $1-p$ and closed with probability $p$ independently. The critical probability $p_{\textnormal{surf}}(d)$ is defined by
 	\[p_{\textnormal{surf}}(d):=\sup\{p\in[0,1]: \mathbb{P}(\text{there is an infinite surface of open plaquettes})>0\},\]
 	where we refer to \cite{GHK14} for the definition of surface. Theorem \ref{thm:per} and Theorem 1.2 of \cite{GHK14} imply that $p_c^b(d) < p_{\textnormal{surf}}(d)$ for each $d\geq 9$.
 \end{remark}

A key step in the proof of Theorem \ref{thm:per} is to prove that $\mathcal{X}_p(d)$ stochastically dominates a site percolation on $\mathbb{Z}^d$ with probability $(1-p)^{2d}$. If we use numerical estimates of $p_c^b(d)$ and $p_c^s(d)$, we can improve the dimension in Theorem \ref{thm:per} to $d \geq 4$. (Preliminary computations suggest that a more sophisticated version of the method in Section \ref{sec:upperbound} and Proposition \ref{prop:per} may extend Theorem \ref{thm:per} to all $d\geq 6$, but we have not verified all details.) The proof in \cite{GHK14} relies on the one-arm critical exponent $\rho(d)<1$; it is believed $ \rho(d)<1$ for all $d\geq 5$ but not for $d=3, 4$. The proof in \cite{BDNS20} uses a second moment method on shielded paths; with numerical estimates of $p_c^b(d)$, their method works for all $d\geq 7$. 

\begin{remark}
	Even though our stochastic domination method should yield a better coexistence result than those of \cite{GHK14,BDNS20}, we believe that this method fails for $d=3$. In other words, we expect that the coexistence of infinite clusters on $\mathbb{Z}^3$ holds for all $p\in (p_c^b(3),p_c^b(3)+\epsilon_3)$ with some $\epsilon_3>0$ but $\mathcal{X}_p(3)$ does NOT stochastically dominate any supercritical site percolation on $\mathbb{Z}^3$. 
\end{remark}

For the Ising model on $\mathbb{Z}^d$ with inverse temperature $\beta \geq 0$ and no external field (see Section~\ref{sec:Ising} for the definition), let $\beta_c(d)$ be the critical inverse temperature. It was proved by Aizenman, Bricmont and Lebowitz \cite{ABL87} that if $d$ is large enough, there is an infinite $+$ cluster in the minus phase (i.e., the infinite-volume measure obtained from all $-$ boundary conditions) for each $\beta \in [0,\beta_c(d)+\epsilon_d)$ where $\epsilon_d>0$. On $\mathbb{Z}^2$, the coexistence of an infinite $+$ cluster and an infinite $-$ cluster does not occur for any $\beta \geq 0$ \cite{Hig82,Hig93}. On $\mathbb{Z}^3$, Campanino and Russo \cite{CR85} proved the coexistence result for all sufficiently small $\beta$; this was generalized to all $\beta$ less than the critical inverse temperature for triangular-type 3D lattices in \cite{JL26}. Our main result on the Ising model is

\begin{theorem}\label{thm:Ising}
	For the Ising model on $\mathbb{Z}^d$ with $d\geq 8$, there is $\epsilon_d>0$ such that for each $\beta \in [0,\beta_c(d)+\epsilon_d)$, the $+$ spins from the minus phase stochastically dominate a site percolation on $\mathbb{Z}^d$ with parameter $\tilde{p}=\tilde{p}(\beta)>p_c^s(d)$. In particular, there is $+$ percolation under the minus phase for all such $\beta$.
\end{theorem}

\begin{remark}
	By the Burton-Keane argument \cite{BK89}, Theorem \ref{thm:Ising} implies that there exist both a unique infinite $+$ cluster and a unique infinite $-$ cluster under the minus phase whenever  $d \geq 8$ and $\beta \in [0,\beta_c(d)+\epsilon_d)$.
\end{remark}

If we use numerical estimates of $\beta_c(d)$ and $p_c^s(d)$, we can improve the dimension in Theorem~\ref{thm:Ising} to all $d\geq 5$.  It is believed that the coexistence result also holds for $d=3, 4$~\cite{ABL87}. However, we expect that the stochastic domination result in Theorem \ref{thm:Ising} fails for $d=3, 4$.

The overall idea of the current paper is to study the coexistence problem via stochastic domination. In Section \ref{sec:percolation}, we prove this for bond percolation in Proposition \ref{prop:per} by a novel application of a recent result by Martineau, Poudevigne and Rax \cite{MPR25}; in Section \ref{sec:Ising}, we prove this for the Ising model in Proposition \ref{prop:Ising} by a stochastic domination result of Liggett and Steif \cite{LS06}. The remaining challenging question is that we need good upper bounds for $p_c^b(d)$ and $p_c^s(d)$. Since the critical probability for oriented percolation has a nice upper bound \cite{CD83}, one usually uses it as an upper bound for $p_c^b(d)$ \cite{BDNS20}. This second moment method (which is attributed to Kesten in \cite{CD83}) only explores all oriented paths on $\mathbb{Z}^d$ and leaves much of the space undetected. Our new strategy is to allow the path to perform a lazy random walk along one coordinate and a positive-oriented walk along the remaining $d-1$ coordinates. For the intermediate dimensions $5 \leq d \leq 10$, our method provides better estimates for both $p_c^b(d)$ and $p_c^s(d)$ than those we could find in the literature  (e.g., \cite{GPS26}); for $d \geq 11$, the upper bound for $p_c^b(d)$ in \cite{FH17} is better than ours. It is clear that a more sophisticated exploration of the space should yield better upper  bounds; we only use 1D random walk because both the analysis and computation are tractable. Our main results, Theorems \ref{thm:per} and \ref{thm:Ising}, are proved in Section \ref{sec:proofofmainthms}. Finally, in Section \ref{sec:numerical}, we list some numerical estimates for $p_c^b(d), p_c^s(d), \beta_c(d)$ that we found in the literature. Combining with Propositions \ref{prop:per} and \ref{prop:Ising}, these values support our claims that the coexistence holds for all $d\geq 4$ in bond percolation and for all $d\geq 5$ in the Ising model.

\section{Percolation and stochastic domination}\label{sec:percolation}
Let $S$ be a finite or countable set, and $Y:=\{Y_x: x \in S\}$, $Z:=\{Z_x: x \in S\}$ be families of random variables taking values in $\{0,1\}$. A function $f: \{0,1\}^S \rightarrow \mathbb{R}$ is \textit{increasing} if $f(\omega) \leq f(\tilde{\omega})$ whenever $\omega \leq \tilde{\omega}$ (i.e., $\omega_x \leq \tilde{\omega}_x$ for each $x \in S$).  An event $A\subset \{0,1\}^S$ is increasing if its indicator function ${\bf1}_A$ is an increasing function. We say that $Y$ \textit{stochastically dominates} Z, written $Y\geq_{\textnormal{st}} Z$, if
\[\mathbb{E} f(Y) \geq \mathbb{E} f(Z)\]
for all bounded, increasing, measurable functions $f$.

We will use the following result on stochastic domination by Martineau, Poudevigne and Rax \cite{MPR25}.
\begin{theorem}[\cite{MPR25}]\label{thm:lift}
	Let $A, B$ be two countable and nonempty sets, and $\pi: A \rightarrow  B$ be a surjective map. Let $(X_b)_{b\in B}$ be a family of i.i.d. Bernoulli random variables with parameter $p\in [0,1]$. On the same probability space, let $(S(b))_{b\in B}$ be a family of random variables such that, for each $b\in B$, $S(b)$ takes values in $\pi^{-1}(\{b\})$. Define $Y=(Y_a)_{a\in A}$ by
	\[Y_a:=X_{\pi(a)} {\bf 1}_{\{S\circ \pi(a)=a\}}.\]
	Let $Z=(Z_a)_{a\in A}$ be i.i.d. Bernoulli random variables with parameter $p$. Then $Y\leq_{\textnormal{st}} Z$.
\end{theorem}

In our intended application of Theorem \ref{thm:lift}, $(X_b)_{b\in B}:=(\omega_e)_{e\in E(\mathbb{Z}^d)}$ is the bond percolation configuration with openness being identified with $1$ and closedness with $0$. For each unoriented edge $xy \in E(\mathbb{Z}^d)$, there are two associated \textit{oriented edges} $\overrightarrow{xy}$ and $\overrightarrow{yx}$; $\pi$ is the map which removes the orientation for each oriented edge (i.e., maps both $\overrightarrow{xy}$ and $\overrightarrow{yx}$ to $xy=yx$). In order to define $(S(b))_{b\in B}$, we first prove
\begin{lemma}\label{lem:1-1}
	For each fixed $\omega \in \{0,1\}^{E(\mathbb{Z}^d)}$, let $\mathcal{C}_{\infty}(\omega)$ be the set of vertices belonging to an infinite open cluster; $\mathcal{C}_{\infty}(\omega):=\emptyset$ if it does not exist.  There is a one-to-one map $f_{\omega}: \{x\in\mathbb{Z}^d: x\in \mathcal{C}_{\infty}(\omega)\} \rightarrow \{e\in E(\mathbb{Z}^d): \omega_e=1\}$ such that $f_{\omega}(x)$ is incident to $x$ for each $x\in \mathcal{C}_{\infty}(\omega)$.
\end{lemma}
\begin{proof}
	We use Marshall Hall's extension of the marriage theorem. Let $U:=\{x\in\mathbb{Z}^d: x\in \mathcal{C}_{\infty}(\omega)\}$, $V:=\{e\in E(\mathbb{Z}^d): \omega_e=1\}$. We build a bipartite graph with bipartite sets $U$ and $V$, and edge set $E_{UV}:=\{xe:  x\in U, e\in E(\mathbb{Z}^d), x\in e, \omega_{e}=1\}$. In this graph, each vertex in $U$ has degree $\leq 2d$. For each finite subset $U_1 \subset U$, there is $m \in \mathbb{N}$ such that $U_1 \subset B_m:=[-m,m]^d \cap \mathbb{Z}^d$. $U_1$ belongs to the union of all open clusters in $\omega|_{B_m}$ (the restriction of $\omega$ to $E(B_m)$). More precisely, there exist distinct open clusters of $\omega |_{B_m}$, $C_1,\dots,C_k$, such that $U_1\cap V(C_j)\neq \emptyset$ for each $j$ and $U_1 \subset \cup_{j=1}^k V(C_j)$. Each such $C_j$ must connect to $B_{m+1} \setminus B_m$ by the definition of $U$. Using a spanning tree of $C_j$ and an open edge between $C_j$ and $B_{m+1} \setminus B_m$, we get
	\[|\{e \in V:  \exists x\in U_1 \cap V(C_j), xe\in E_{UV}\}| \geq |U_1 \cap V(C_j)|,\ \forall j=1,\dots,k.\]
	Summing over $j$ gives 
	\[|\{e \in V:  \exists x \in U_1, xe\in E_{UV}\}| \geq |U_1|.\]
	The lemma follows from Theorem 1 of \cite{Hal48}.
\end{proof}

%For each $xy\in E(\mathbb{Z}^d)$, define (note that $S$ depends on $\omega$)
%\begin{equation}\label{eq:Sdef}
%	S(xy):=\begin{cases}
%		\overrightarrow{xy}, & x\in \mathcal{C}_{\infty}(\omega), xy=f_{\omega}(x),\\
% 	\overrightarrow{yx}, & y\in \mathcal{C}_{\infty}(\omega), yx=f_{\omega}(y),\\
%		\text{arbitrary } \overrightarrow{xy} \text{ or }\overrightarrow{yx}, & \text{otherwise}.
%	\end{cases}
%\end{equation}

The following proposition is our main stochastic domination result for percolation.
\begin{proposition}\label{prop:per}
	For bond percolation on $\mathbb{Z}^d$ with edge open probability $p$, $V(\mathbb{Z}^d) \setminus \mathcal{C}_{\infty}$ stochastically dominates a site percolation with parameter $(1-p)^{2d}$.
\end{proposition}
\begin{proof}
	The choice of $f_{\omega}$ in Lemma \ref{lem:1-1} is not a priori measurable in $\omega$. So we prove the stochastic domination in each finite box $B_n$ first. Let
	\[B:=\{e \in E(\mathbb{Z}^d): e \cap B_n \neq \emptyset\}.\]
	So $B$ includes those edges crossing the boundary of $B_n$. Let $A:=\cup_{xy \in B}\{\overrightarrow{xy},\overrightarrow{yx}\}$ and $X_e:=\omega_e$ for each $e \in B$. We claim that there exist random variables $(S_e)_{e\in B}$ which assign exactly one direction  to each $e \in B$ such that: 
	\[\forall x \in  B_n \cap \mathcal{C}_{\infty}(\omega), \exists y \sim x: \omega_{xy}=1 \text{ and } S_{xy}(\omega)=\overrightarrow{xy},\]
	where $y \sim x$ means $xy \in E(\mathbb{Z}^d)$. We next prove the proposition modulo the proof of this claim.
	 
	 Define $Y=(Y_a)_{a\in A}$ by
	\[Y_{\overrightarrow{xy}}(\omega)=\omega_{xy}{\bf 1}_{\{S_{xy}=\overrightarrow{xy}\}}, \ \forall \omega \in \{0,1\}^{E(\mathbb{Z}^d)}.\]
	Then Theorem \ref{thm:lift} implies that $Y\leq_{\textnormal{st}} Z$ where $Z=(Z_{\overrightarrow{xy}})_{\overrightarrow{xy}\in A}$ is a family of i.i.d. Bernoulli($p$) random variables. The definition of $S_e$ implies that
	\begin{equation}\label{eq:indlesY}
		{\bf 1}_{\{x\in \mathcal{C}_{\infty}(\omega)\}} \leq \max_{y \sim x} Y_{\overrightarrow{xy}}(\omega), \ \forall x\in B_n, \omega \in \{0,1\}^{E(\mathbb{Z}^d)}.
	\end{equation}
	$Y\leq_{\textnormal{st}} Z$ implies that there exists a probability measure $\nu$ on $\{0,1\}^{A} \times \{0,1\}^{A}$ with marginals distributed as $Y$ and $Z$ such that
	\[\nu(\{(\omega,\tilde{\omega}): \omega_{\vec{e}}\leq \tilde{\omega}_{\vec{e}}, \forall \vec{e}\in A\})=1.\]
	So we have
	\begin{equation}\label{eq:YleqZ}
		\max_{y \sim x} Y_{\overrightarrow{xy}} \leq \max_{y \sim x} Z_{\overrightarrow{xy}}\  \nu \text{-a.s.}, \ \forall x\in B_n.
	\end{equation}
	Note that $(\max_{y \sim x} Z_{\overrightarrow{xy}})_{x\in B_n}$ is a family of i.i.d. Ber$(1-(1-p)^{2d})$ random variables. This combined with \eqref{eq:indlesY} and \eqref{eq:YleqZ} proves the stochastic domination restricted to $B_n$. Since $n$ is arbitrary, the desired domination on $\mathbb{Z}^d$ follows. 
	
	Let us now prove the claim.
	Enumerate all pairs 
	\[(D_j,g_j), \ j=1,\dots, M\]
	where $D_j \subset B_n$ and $g_j: D_j \rightarrow B$ is one-to-one and assigns each vertex in $D_j$ an incident edge. For each $j$, define
	\[F_j:= \bigcap_{x \in D_j} \{\omega: x \in \mathcal{C}_{\infty}(\omega)\} \bigcap \bigcap_{x \in B_n \setminus D_j} \{\omega: x \notin \mathcal{C}_{\infty}(\omega)\} \bigcap \bigcap_{x\in D_j}\{\omega: \omega_{g_j(x)}=1\}.\]
	That is, $F_j$ is the event that $(D_j,g_j)$ is admissible. It is easy to see that each $F_j$ is measurable. For each fixed $\omega$, Lemma \ref{lem:1-1} implies that at least one $(D_j,g_j)$ is admissible. Let
	\[E_j:=F_j \setminus \cup_{k=1}^{j-1} F_k, \ j=1,\dots,M\]
	Then $(E_j)_{j=1}^M$forms a measurable partition of the full configuration space; $E_j$ is the event that $(D_j,g_j)$ is the first admissible candidate.
	
	For each $j$, let $s_j$ be the following deterministic orientation of edges in $B$: for each $x\in D_j$, orient $g_j(x)$ away from $x$, and orient the remaining edges in $B$ by a fixed rule. Define
	\[S_e(\omega):=s_j(e), \text{ if }\omega\in E_j.\]
	Then for each $e\in B$ and either orientation $\vec{f}$ of $e$,
	\[\{\omega: S_e(\omega)=\vec{f}\}=\cup_{j, s_j(e)=\vec{f}} E_j,\]
	which is a union of measurable events. This finishes the proof of the claim and thus the proposition.
\end{proof}

\begin{remark}
	One may get a better parameter than  $(1-p)^{2d}$ in Proposition \ref{prop:per}. However, the improvement that we obtained is too small to strengthen Theorem \ref{thm:per}. 
\end{remark}

\section{Ising model and stochastic domination}\label{sec:Ising}
In this section, we prove a result analogous to Proposition \ref{prop:per} for the Ising model. Let $\Lambda \subset \mathbb{Z}^d$ be a finite subset; we view $\Lambda$ as a subgraph of $\mathbb{Z}^d$ with vertex set $\Lambda$ and edge set $E(\Lambda):=\{xy \in E(\mathbb{Z}^d): x, y\in \Lambda\}$. The Ising model on $\Lambda$ at inverse temperature $\beta \geq 0$ with boundary conditions $\eta \in \{-1, 0, +1\}^{\mathbb{Z}^d}$ is the probability measure $P_{\Lambda,\beta}^{\eta}$ on $\{-1,+1\}^{\Lambda}$ such that
\begin{equation}\label{eq:Isingdef}
	P_{\Lambda,\beta}^{\eta}(\sigma):=\frac{\exp[\beta \sum_{xy\in E(\Lambda)}\sigma_x\sigma_y+\beta \sum_{xy: x \sim y, x\in \Lambda, y\notin \Lambda}\sigma_x\eta_y]}{Z_{\Lambda,\beta}^{\eta}}, \ \forall \sigma \in \{-1,+1\}^{\Lambda},
\end{equation}
where $Z_{\Lambda,\beta}^{\eta}$ is the partition function. We write $P^-_{\Lambda,\beta}$ ($P^0_{\Lambda,\beta}$, respectively) if $\eta_x=-1$ ($\eta_x=0$, respectively) for each $x\in \mathbb{Z}^d$. By the GKS inequalities \cite{Gri67,KS68}, $P^0_{\Lambda,\beta}$ converges weakly to an infinite-volume measure $P^0_{\mathbb{Z}^d,\beta}$ as $\Lambda \uparrow \mathbb{Z}^d$; the FKG inequality implies that $P^-_{\Lambda,\beta}$ converges weakly to an infinite-volume measure $P^-_{\mathbb{Z}^d,\beta}$ as $\Lambda \uparrow \mathbb{Z}^d$, which we call the \textit{minus phase}. When we talk about stochastic domination for the Ising model, we may view $+1$ as $1$ and $-1$ as $0$. Liggett and Steif \cite{LS06} proved the following result.

\begin{theorem}[\cite{LS06}]\label{thm:downwardFKG}
	For the minus phase $P^-_{\mathbb{Z}^d,\beta}$ of the Ising model, the following two statements are equivalent:
	\begin{enumerate}[(i)]
		\item $+$ spins under $P^-_{\mathbb{Z}^d,\beta}$ stochastically dominate a site percolation on $\mathbb{Z}^d$ with parameter $p\in[0,1]$.
		\item $P^-_{\mathbb{Z}^d,\beta}(\sigma_x=-1, \forall x \in [1,n]^d \cap \mathbb{Z}^d)\leq (1-p)^{n^d}$ for all sufficiently large $n$.
	\end{enumerate}
\end{theorem}
\begin{proof}
	Since $P^-_{\Lambda,\beta}$ satisfies the FKG lattice condition for each $\Lambda$, $P^-_{\Lambda,\beta}$ has the downward FKG property defined in \cite{LS06}.  By sending $\Lambda \uparrow \mathbb{Z}^d$, $P^-_{\mathbb{Z}^d,\beta}$ also satisfies the downward FKG property. It is clear that $P^-_{\mathbb{Z}^d,\beta}$ is translation invariant. So Theorem 4.1 of \cite{LS06} and Remark (1) right after it complete the proof of the Theorem. Note that Theorem 4.1 of \cite{LS06} is stated for $\mathbb{Z}^2$ but as mentioned before this theorem, it can be extended to $\mathbb{Z}^d$.
\end{proof}

Let $f_{\mathbb{Z}^d}(\beta)$ be the \textit{free energy} (or pressure) defined by
\begin{equation}\label{eq:freeenergydef}
	f_{\mathbb{Z}^d}(\beta):=\lim_{n\rightarrow \infty} \frac{\ln Z_{B_n,\beta}^{\eta}}{|B_n|},
\end{equation}
where we recall that $B_n:=[-n,n]^d \cap \mathbb{Z}^d$. It is well-known that this limit exists and is independent of $\eta$; see, e.g., Theorem 3.6 of \cite{FV18}. A simple computation (see Corollary 4.2 of \cite{LS06}) gives
\begin{equation}
	\lim_{n\rightarrow\infty} \left[P^-_{\mathbb{Z}^d,\beta}(\sigma_x=-1, \forall x \in [1,n]^d \cap \mathbb{Z}^d)\right]^{1/n^d}=\frac{e^{d\beta}}{\exp[f_{\mathbb{Z}^d}(\beta)]}.
\end{equation}
In order to apply Theorem \ref{thm:downwardFKG}, we need to find an upper bound for this limit.
\begin{lemma}\label{lem:freeenergy}
	Let $e_1:=(1,0,\dots,0)$ be the first coordinate unit vector, $\langle \cdot \rangle^0_{\mathbb{Z}^d,t}$ be the expectation with respect to the free measure $P^0_{\mathbb{Z}^d,t}$. Then
	\[f_{\mathbb{Z}^d}(\beta)=\ln 2+d \int_0^{\beta} \langle \sigma_0 \sigma_{e_1} \rangle^0_{\mathbb{Z}^d,t} \, dt \geq \ln 2+ d\ln \cosh(\beta).\]
\end{lemma}
\begin{proof}
	Let $Z^0_{B_n,\beta}$ be the partition function for the Ising model on $B_n$ with free boundary conditions, and 
	\[f_{B_n}(\beta):=\frac{\ln Z^0_{B_n,\beta}}{|B_n|}.\]
	Then
	\[f'_{B_n}(\beta)=\frac{\sum_{xy\in E(B_n)}\langle \sigma_x\sigma_y \rangle^0_{B_n,\beta}}{|B_n|},\]
	where $\langle \cdot \rangle^0_{B_n,\beta}$ is the expectation with respect to $P^0_{B_n,\beta}$. Therefore,
	\begin{equation}\label{eq:freeint}
		f_{B_n}(\beta)-f_{B_n}(0)=\frac{1}{|B_n|}\sum_{xy \in E(B_n)}\int_0^{\beta} \langle \sigma_x \sigma_y \rangle^0_{B_n,t} \, dt.
	\end{equation}
	The GKS inequalities give $\langle \sigma_x\sigma_y \rangle^0_{B_n,t} \leq \langle \sigma_x\sigma_y \rangle^0_{\mathbb{Z}^d,t}=\langle \sigma_0 \sigma_{e_1} \rangle^0_{\mathbb{Z}^d,t}$, and thus
	\begin{equation}\label{eq:freeub}
		\limsup_{n\rightarrow \infty}\frac{1}{|B_n|}\sum_{xy \in E(B_n)}\int_0^{\beta} \langle \sigma_x \sigma_y \rangle^0_{B_n,t} \, dt \leq d \int_0^{\beta}\langle \sigma_0 \sigma_{e_1} \rangle^0_{\mathbb{Z}^d,t} \, dt.
	\end{equation}
	Another application of the GKS inequalities gives $\langle \sigma_x\sigma_y \rangle^0_{B_n,t}  \geq \langle \sigma_0\sigma_{e_1} \rangle^0_{B_M,t}$ for each $xy\in E(B_{n-2M})$ with $M\in \mathbb{N}$. Hence,
	\begin{align*}
		\liminf_{n\rightarrow \infty}\frac{1}{|B_n|}\sum_{xy \in E(B_n)}\int_0^{\beta} \langle \sigma_x \sigma_y \rangle^0_{B_n,t} \, dt &\geq \liminf_{n\rightarrow\infty} \frac{|E(B_{n-2M})|}{|B_n|}\int_0^{\beta}\langle \sigma_0 \sigma_{e_1} \rangle^0_{B_M,t} \, dt\\
		&=d \int_0^{\beta}\langle \sigma_0 \sigma_{e_1} \rangle^0_{B_M,t} \, dt. 
	\end{align*}
	Taking $M\rightarrow\infty$ and combining with \eqref{eq:freeenergydef}, \eqref{eq:freeint}, \eqref{eq:freeub} yield
	\[f_{\mathbb{Z}^d}(\beta)=\ln 2+d \int_0^{\beta} \langle \sigma_0 \sigma_{e_1} \rangle^0_{\mathbb{Z}^d,t} \, dt.\]
	Let $K_2$ be the graph with vertex set $\{0,e_1\}$ and edge set $\{0e_1\}$. Then the GKS inequalities imply
	\[ \langle \sigma_0 \sigma_{e_1} \rangle^0_{\mathbb{Z}^d,t} \geq  \langle \sigma_0 \sigma_{e_1} \rangle^0_{K_2,t}=\tanh(t).\]
	Therefore,
	\[\int_0^{\beta} \langle \sigma_0 \sigma_{e_1} \rangle^0_{\mathbb{Z}^d,t} \, dt \geq \int_0^{\beta} \tanh(t) \, dt =\ln \cosh(\beta).\]
	This completes the proof of the lemma.
\end{proof}
Combining Theorem \ref{thm:downwardFKG} and Lemma \ref{lem:freeenergy}, we get that $+$ spins under $P^-_{\mathbb{Z}^d,\beta}$ stochastically dominate a site percolation on $\mathbb{Z}^d$ of parameter $\max\{1-e^{d\beta}/(2[\cosh(\beta)]^d)-\epsilon, 0\}$ for each small $\epsilon$. Taking $\epsilon\rightarrow 0$, we obtain
\begin{proposition}\label{prop:Ising}
	$+$ spins under $P^-_{\mathbb{Z}^d,\beta}$ stochastically dominate a site percolation on $\mathbb{Z}^d$ with parameter $\max\{1-\frac{e^{d\beta}}{2[\cosh(\beta)]^d}, 0\}$.
\end{proposition}
\begin{remark}
	By using a larger graph than $K_2$ in the proof of Lemma \ref{lem:freeenergy}, one can get a better estimate for $f_{\mathbb{Z}^d}(\beta)$ and improve Proposition \ref{prop:Ising}.
\end{remark}

\section{Upper bounds for critical probabilities}\label{sec:upperbound}
In this section, we use a second moment method to obtain upper bounds for critical probabilities. This method was used by Cox and Durrett \cite{CD83} in their study on oriented percolation. It is also a central tool in \cite{BDNS20}; see also \cite{BPP98} for further developments.

We fix $d\geq 4$, and write
\[\mathbb{Z}^d=\mathbb{Z}\times \mathbb{Z}^{\tilde{d}} \text{ with }\tilde{d}:=d-1.\] 
We call the first coordinate \textit{horizontal}, and the other $\tilde{d}$ coordinates \textit{outer} coordinates. We construct self-avoiding paths by performing a lazy random walk on the horizontal coordinate and a positive-oriented walk along the outer coordinates. More precisely, for $q\in(0,1/4]$, define
\[\pi_q(-1)=\pi_q(1):=q,\ \pi_q(0)=1-2q.\]
Starting from $(X_0,Y_0)=(0,0)\in \mathbb{Z}\times \mathbb{Z}^{\tilde{d}}$, at macrostep $n\in \mathbb{N}$, choose independently
\[K_n \sim \pi_q, I_n= \text{uniform on }\{1,2,\dots, \tilde{d}\}.\]
Then define
\[X_n:=X_{n-1}+K_n, Y_n:=Y_{n-1}+e_{I_n},\]
where $e_1,\dots,e_{\tilde{d}}$ are the standard basis vectors for $\mathbb{Z}^{\tilde{d}}$.

At macrostep $n$, if $K_n\neq 0$, the path first traverses the horizontal edge $(X_{n-1},Y_{n-1})\rightarrow (X_{n-1}+K_n,Y_{n-1})$ and then traverses the outer edge $(X_{n-1}+K_n,Y_{n-1})\rightarrow (X_{n-1}+K_n, Y_{n-1}+e_{I_n})$; if $K_n=0$, the path only traverses the outer edge. Let $\mu_q$ be the resulting measure on infinite paths. During macrostep $n$, the starting vertex and possible horizontal endpoint have the sum of outer coordinates equaling $n-1$, and its ending vertex increases that sum to $n$. Hence, concatenating the macrosteps produces a self-avoiding nearest-neighbor path, with consecutive macrosteps sharing only one common vertex. Let $\gamma:=(\gamma_0,\gamma_1,\dots) \sim \mu_q$ such that $\gamma_n$ represents the position after $n$ macrosteps. For bond percolation with parameter $p$, we define
\[Z_n^b:=\int p^{-|E(\gamma[0,n])|} {\bf 1}_{\{E(\gamma[0,n]) \text{ is open}\}} \, d \mu_q(\gamma),\] 
where $\gamma[0,n]$ is the full nearest-neighbor path through $n$ macrosteps.
Then
\[\mathbb{E}^b_p [Z_n^b]=1, \ \forall n\in \mathbb{N}.\]
For two independent paths $\gamma$, $\gamma' \sim \mu_q$,
\begin{equation}\label{eq:bond2ndmon}
	\mathbb{E}^b_p[(Z_n^b)^2]=\mathbb{E}_{\mu_q\otimes \mu_q}[p^{-|E(\gamma[0,n]) \cap E(\gamma'[0,n])|}].
\end{equation}
Similarly, for site percolation,
\[Z_n^s:=\int p^{-|V(\gamma[0,n])|} {\bf 1}_{\{V(\gamma[0,n]) \text{ is open}\}} \, d \mu_q(\gamma)\] 
satisfies
\begin{equation}\label{eq:site2ndmon}
	\mathbb{E}^s_p [Z_n^s]=1, \mathbb{E}^s_p[(Z_n^s)^2]=\mathbb{E}_{\mu_q\otimes \mu_q}[p^{-|V(\gamma[0,n]) \cap V(\gamma'[0,n])|}].
\end{equation}
Let $\gamma=(X_n,Y_n)_{n\in \mathbb{N}\cup \{0\}}$ and $\gamma'=(X'_n,Y'_n)_{n\in \mathbb{N}\cup \{0\}}$, and denote their choices at macrostep $n$ by $(K_n,I_n)$ and $(K'_n,I'_n)$. Define
\[D_n^X:=X_n-X'_n, D_n^Y:=Y_n-Y'_n, S_n:=(D_n^X,D_n^Y), \ \forall n\in \mathbb{N}\cup\{0\}.\]
Then
\[(D_0^X, D_0^Y)=(0,0), D_n^X=D_{n-1}^X+K_n-K'_n, D_n^Y=D_{n-1}^Y+e_{I_n}-e_{I'_n}, \ \forall n\in \mathbb{N}.\]
So $(D_n^X)_{n\in \mathbb{N}\cup \{0\}}$ and $(D_n^Y)_{n\in \mathbb{N}\cup \{0\}}$ are independent random walks. Most of this section is devoted to the study of the random walk $(S_n)_{n\in \mathbb{N}\cup \{0\}}$. It is clear that $(S_n)$ lives in the space
\[\mathcal{I}:=\mathbb{Z}\times\{z \in \mathbb{Z}^{\tilde{d}}: \sum_{j=1}^{\tilde{d}} z_j=0\}.\]

We first write the number of intersections in \eqref{eq:bond2ndmon} and \eqref{eq:site2ndmon} as a sum over different macrosteps. Let 
\[H(k):=\begin{cases}
	\{0\}, & k=0,\\
	\{0,k\}, & k \in \{-1,1\}.
\end{cases}\]
Let
\[r_a(k,k'):=|H(k) \cap (-a+H(k'))|, a \in \mathbb{Z}, k, k'\in \{-1,0,1\}\]
be the number of common vertices at outer height $0$ during macrostep $1$ if $\gamma_0=0, K_1=k, \gamma'_0=(-a,0), K'_1=k'$. Note that
\[V(\gamma[0,n])=\left(\bigsqcup_{j=1}^n\{(X_{j-1}+h,Y_{j-1}): h\in H(K_j)\}\right) \bigsqcup \{(X_n,Y_n)\},\]
where the unions are disjoint because all sets have different outer heights. Therefore,
\begin{equation}\label{eq:vertexint}
	|V(\gamma[0,n]) \cap V(\gamma'[0,n])|=\sum_{j=1}^n N_j^s+{\bf 1}_{\{S_n=0\}},
\end{equation}
where
\begin{equation}\label{eq:N^sdef}
	N_j^s:=\begin{cases}
		r_{D_{j-1}^X}(K_j,K'_j), & D_{j-1}^Y=0,\\
		0, & D_{j-1}^Y\neq 0.
	\end{cases}
\end{equation}

Let
\[\epsilon_a(k,k'):={\bf 1}_{\{k \neq 0, k' \neq 0, \text{ the edge with endpoints } 0 \text{ and }k =\text{ the edge with endpoints } -a \text{ and }-a+k' \}}\]
be the number of common horizontal edges at macrostep $1$ if $\gamma_0=0, K_1=k, \gamma'_0=(-a,0), K'_1=k'$. Then we have
\begin{equation}\label{eq:edgeint}
	|E(\gamma[0,n]) \cap E(\gamma'[0,n])|=\sum_{j=1}^n N_j^b,
\end{equation}
where
\begin{equation}\label{eq:N^bdef}
	N_j^b:=\begin{cases}
		\epsilon_{D_{j-1}^X}(K_j,K'_j), & D_{j-1}^Y=0, I_j \neq I'_j,\\
		\epsilon_{D_{j-1}^X}(K_j,K'_j)+{\bf 1}_{\{D^X_{j-1}=K'_j-K_j\}}, & D_{j-1}^Y=0, I_j = I'_j,\\
		0, & D_{j-1}^Y\neq 0.
		\end{cases}
\end{equation}

We next compute the Green function for $(S_n)_{n\in \mathbb{N}\cup \{0\}}$. For the horizontal random walk, define
\begin{equation}\label{eq:xdef}
	x_n(r):=\mathbb{P}(D_n^X=r), \ \forall r\in \mathbb{Z}, n \in \mathbb{N}\cup \{0\}.
\end{equation}
For the outer random walk, define
\begin{equation}\label{eq:uvdef}
	u_n:=\mathbb{P}_0(D_n^Y=0), v_n:=\mathbb{P}_{e_1-e_2}(D_n^Y=0), \ \forall n\in \mathbb{N}\cup \{0\}.
\end{equation}
Let $P$ be the transition kernel of $(S_n)_{n\geq 0}$, and $G:=\sum_{n=0}^{\infty}P^n$ be its Green kernel. Then we have
\[G((a,0),(b,0))=g(b-a), G((a,e_i-e_j),(b,0))=h(b-a),\ \forall a, b\in \mathbb{Z}, i \neq j \in \{1,\dots, \tilde{d}\},\]
where (by independence)
\begin{equation}\label{eq:ghdef}
	g(r):=\sum_{n=0}^{\infty} u_n x_n(r), h(r):=\sum_{n=0}^{\infty} v_n x_n(r), \ \forall r \in \mathbb{Z}.
\end{equation}
(Note that both sums are finite by \eqref{eq:u_nx_nbound} below and the assumption $d\geq 4$.)

We will see that the intersection property of $(S_n)_{n\in \mathbb{N}\cup \{0\}}$ is totally determined by two $5\times 5$ matrices that we define now. Let 
\begin{equation}\label{eq:QCdef}
	\mathcal{S}:=\{-2,-1,0,1,2\}, \mathcal{C}:=\{(a,0): a \in \mathcal{S}\}.
\end{equation}
Define two matrices
\begin{align}
	&M_p^b(a,b):=\sum_{k,l=-1}^1 \pi_q(k) \pi_q(l)\Big[\frac{1}{\tilde{d}}(p^{-\epsilon_a(k,l)-{\bf 1}_{\{a=l-k\}}}-1)g(b-(a+k-l))\nonumber\\
	 &\qquad +\frac{\tilde{d}-1}{\tilde{d}}(p^{-\epsilon_a(k,l)}-1)h(b-(a+k-l))\Big],\label{eq:M^bdef}\\
	&M_p^s(a,b):=\sum_{k,l=-1}^1\pi_q(k) \pi_q(l)(p^{-r_a(k,l)}-1)\Big[\frac{1}{\tilde{d}}g(b-(a+k-l))\nonumber\\
	 &\qquad +\frac{\tilde{d}-1}{\tilde{d}}h(b-(a+k-l))\Big], \ \forall a, b \in \mathcal{S}.\label{eq:M^sdef}
\end{align}
For a square matrix $M$, let $\rho(M)$ be its spectral radius. The following theorem provides estimates on critical probabilities through studying the above two matrices.
\begin{theorem}\label{thm:p_cbound}
	Suppose $d\geq 4$ and $p \in (0,1]$. If $\rho(M_p^b)<1$ then $p_c^b(d)\leq p$; if $\rho(M_p^s)<1$ then $p_c^s(d)\leq p$.
\end{theorem}
\begin{proof}
	Let $N_1^*$ be either $N_1^b$ or $N_1^s$ defined in \eqref{eq:N^bdef} and \eqref{eq:N^sdef} respectively. Define the weighted transition kernel
	\[\hat{P}_p(x,y):=\mathbb{E}_x[p^{-N_1^*}{\bf 1}_{\{S_1=y\}}], \ \forall x,y \in \mathcal{I}.\]
	Let 
	\[Q_p:=\hat{P}_p-P.\]
	Since $N_1^*\geq 0$, $Q_p\geq 0$ (i.e., each entry is nonnegative). A key observation is that the rows of $Q_p$ vanish outside $\mathcal{C}$ (defined in \eqref{eq:QCdef}); this is because by \eqref{eq:N^bdef} and \eqref{eq:N^sdef}, $Q_p(x,y)=0$ if the outer component of $x$ is nonzero or its horizontal component does not belong to $S$. The Markov property gives
	\begin{equation}\label{eq:hatP^n}
		\hat{P}_p^n \vec{1}(x)=\mathbb{E}_x[p^{-\sum_{j=1}^n N_j^*}], \ \forall x\in \mathcal{I}, \forall n \in \mathbb{N}.
	\end{equation}
	Define
	\[M(\vec{a},\vec{b}):=\sum_{y \in \mathcal{I}} Q_p(\vec{a},y) G(y, \vec{b}),\ \vec{a}:=(a,0)\in \mathcal{C}, \vec{b}:=(b,0)\in \mathcal{C},\]
	\[\eta(\vec{a}):=\sum_{y \in \mathcal{I}} Q_p(\vec{a},y),\ \vec{a}=(a,0)\in \mathcal{C}.\]
	It is easy to check that \eqref{eq:M^bdef} and \eqref{eq:M^sdef} agree with this definition upon the identification $\vec{a} \rightarrow a$; so $M$ is either $M^b_p$ or $M^s_p$ under this identification. If we expand
	\[\hat{P}_p^n=(Q_p+P)^n,\]
	a term containing $m\geq 1$ copies of $Q_p$ has the form
	\[P^{l_0}Q_pP^{l_1}Q_p\dots Q_p P^{l_m}, \ l_0+l_1+\dots+l_m+m=n.\]
	Since all matrices are nonnegative, we may replace $P^{l_j}$ by $G$ to get an upper bound
	\begin{align*}
		\hat{P}_p^n \vec{1}(0)& \leq 1+\sum_{m=1}^{\infty} \sum_{\vec{a} \in \mathcal{C}} G(0,\vec{a})(Q_pGQ_p\dots Q_p\vec{1})(\vec{a})\\
		&=1+ G_{0,\mathcal{C}} \sum_{m=1}^{\infty} M^{m-1}\eta,
	\end{align*}
	where we have used $P^{l_m}\vec{1}=\vec{1}$ in the first inequality and the key observation above \eqref{eq:hatP^n} in the equality, and $G_{0,\mathcal{C}} $ is the row vector $(G(0,\vec{a}):\vec{a} \in \mathcal{C})$. If $\rho(M)<1$, then
	\[\sum_{m=1}^{\infty} M^{m-1}=(I-M)^{-1} \text{ is finite}.\]
	Therefore, 
	\[\sup_n \hat{P}^n_p \vec{1}(0)\leq 1+ G_{0,\mathcal{C}} (I-M)^{-1}\eta < \infty.\]
	For bond percolation, \eqref{eq:bond2ndmon}, \eqref{eq:edgeint}, \eqref{eq:hatP^n} and the last inequality give a uniform bound on $\mathbb{E}^b_p[(Z_n^b)^2]$. For site percolation, \eqref{eq:site2ndmon}, \eqref{eq:vertexint}, \eqref{eq:hatP^n} and the last displayed inequality give the same conclusion (i.e., the same upper bound with an extra factor $p^{-1}$).
	
	Therefore, the Cauchy-Schwarz inequality implies
	\[\mathbb{P}^*_{p}(Z_n^*>0)\geq \frac{(\mathbb{E}^*_p[Z_n^*])^2}{\mathbb{E}^*_p[(Z_n^*)^2]}>c>0, \ \forall n\in \mathbb{N}.\]
	Since $\{Z_n^*>0\}$ is decreasing in $n$, we have
	\[\mathbb{P}^*_{p}(\cap_{n=1}^{\infty}\{Z_n^*>0\}) \geq c>0.\]
	Since the macropath tree is finitely branching, K\"{o}nig's infinity lemma implies that there is an infinite open path with probability at least $c$.
\end{proof}

Note that $M_p^*$ in \eqref{eq:M^bdef} and \eqref{eq:M^sdef} actually depends on $q$. An immediate consequence of Theorem \ref{thm:p_cbound} is
\[p^*_c(d) \leq \inf_{0< q \leq 1/4}\inf\{p\in(0,1]: \rho(M_p^*)<1\},\ * = b\text{ or }s.\]
For simplicity, we choose
\begin{equation}\label{eq:qdef}
	q=\frac{1}{d+3}.
\end{equation}

Recall that Gelfand's formula says
\[\rho(M)=\lim_{k\rightarrow \infty} \|M^k\|^{1/k},\]
where $\|\cdot\|$ is the Frobenius norm. If $0\leq M \leq \tilde{M}$ (entrywise inequality), then
\begin{equation}\label{eq:srcom}
	\rho(M) \leq \rho(\tilde{M}).
\end{equation}
A min-max Collatz-Wielandt formula gives that for any $M \geq 0$,
\begin{equation}\label{eq:rhoupperbound}
	\rho(M) \leq \max_i \frac{(Mw)_i}{w_i}, \ \forall w>0.
\end{equation}
Since $M_p^*(a,b)=M_p^*(-a,-b)$ for all $a,b \in \mathcal{S}$, we choose a positive vector of the form
\begin{equation}\label{eq:wdef}
	w^*=(w^*_2,w^*_1,1,w^*_1,w^*_2)'.
\end{equation}

We need some upper bounds for $g$ and $h$ defined in \eqref{eq:ghdef}. Recall $u_n$ and $v_n$ in \eqref{eq:uvdef}. It is clear that
\begin{equation}\label{eq:u_ncomp}
	u_n=\frac{1}{\tilde{d}^{2n}}\sum_{k_1+\dots+k_{\tilde{d}}=n, k_i \geq 0}\binom{n}{k_1,\dots,k_{\tilde{d}}}^2.
\end{equation}
The Cauchy-Schwarz inequality gives
\begin{equation}\label{eq:vlequ}
	v_n \leq u_n, \ \forall n\in \mathbb{N}\cup\{0\}.
\end{equation}
Conditioning on the first step $D_1^Y$ gives
\begin{equation}\label{eq:v_ncomp}
	u_{n+1}=\frac{1}{\tilde{d}}u_n+\frac{\tilde{d}-1}{\tilde{d}}v_n, \ \forall n\in \mathbb{N}\cup\{0\}.
\end{equation}
Robbins' bounds
\[\sqrt{2\pi}k^{k+1/2} e^{-k} < k! <\sqrt{2\pi} k^{k+1/2} e^{-k} e^{1/(12k)}, \ k\in \mathbb{N}\]
imply
\begin{equation}\label{eq:uupperbound}
	u_n \leq \frac{1}{\tilde{d}^{n}}\max_{k_1+\dots+k_{\tilde{d}}=n, k_i\geq 0} \binom{n}{k_1,\dots,k_{\tilde{d}}} \leq \frac{\tilde{d}^{\tilde{d}/2}}{(2\pi)^{(\tilde{d}-1)/2}} e^{1/(12n)} \left(1-\frac{\tilde{d}}{n}\right)^{-\tilde{d}/2} n^{-(\tilde{d}-1)/2}, \ \forall n> \tilde{d},
\end{equation}
where we have used that the maximum occurs at a balanced vector $|k_i-k_j|\leq 1$ for each $i,j$ and thus $k_i \geq n/\tilde{d} -1$ for each $i$. The Fourier inversion formula and \eqref{eq:xdef} give
\[x_n(r)=\frac{1}{2\pi}\int_{-\pi}^{\pi} e^{-irt}(1-2q+2q \cos(t))^{2n} \, dt=\frac{1}{2\pi}\int_{-\pi}^{\pi} e^{-irt}(1-4q \sin^2(t/2))^{2n} \, dt.\]
Using
\[1-x\leq e^{-x}, \  \sin(|t|/2) \geq |t|/\pi, \ \forall |t|\leq \pi,\]
we get that for all $r\in\mathbb{Z}$,
\begin{equation*}\label{eq:xupperbound}
	x_n(r) \leq \frac{1}{2\pi}\int_{-\pi}^{\pi} \exp[-8nqt^2/\pi^2]\, dt \leq \frac{1}{2\pi}\int_{-\infty}^{\infty} \exp[-8nqt^2/\pi^2]\, dt = \frac{\sqrt{\pi}}{4\sqrt{2nq}}, \ \forall n \in \mathbb{N}.
\end{equation*}
Combining this with \eqref{eq:uupperbound}, we get that for each integer $N\geq \tilde{d}$,
\begin{equation}\label{eq:u_nx_nbound}
	u_n x_n(r) \leq C_{d,N} n^{-\tilde{d}/2}, \ \forall r \in \mathbb{Z}, n \geq N+1,
\end{equation}
where
\[C_{d,N}:=\frac{\tilde{d}^{\tilde{d}/2}}{(2\pi)^{(\tilde{d}-1)/2}} e^{1/(12(N+1))} \left(1-\frac{\tilde{d}}{N+1}\right)^{-\tilde{d}/2}\frac{\sqrt{\pi}}{4\sqrt{2q}}.\]
By \eqref{eq:vlequ}, the same bound holds with $u_n$ replaced by $v_n$. Therefore,
\[\sum_{n>N} u_n x_n(r) \leq C_{d,N} \sum_{n>N} n^{-\tilde{d}/2} \leq C_{d,N} \frac{2 N^{1-\tilde{d}/2}}{\tilde{d}-2}=: \tau_{d,N}, \ \sum_{n >N} v_n x_n(r) < \tau_{d,N}.\]
To estimate $p_c^b(d)$ and $p_c^s(d)$ for $4\leq d \leq 15$, we choose $N:=500$. Then
\[g(r)\leq \bar{g}(r), \ h(r) \leq \bar{h}(r),\]
where
\[\bar{g}(r):=\sum_{n=0}^{500} u_n x_n(r) + \tau_{d,500}, \ \bar{h}(r):=\sum_{n=0}^{500} v_n x_n(r)+\tau_{d,500}.\]
From \eqref{eq:xdef}, \eqref{eq:qdef}, \eqref{eq:u_ncomp} and \eqref{eq:v_ncomp}, one can see that $x_n(r), u_n, v_n$ are all rational numbers. Since both $g$ and $h$ are even in $r$, only $0\leq r \leq 6$ are needed in \eqref{eq:M^bdef} and \eqref{eq:M^sdef}. Let $\overline{M}_p^*$ be the matrix obtained from \eqref{eq:M^bdef} and \eqref{eq:M^sdef} with $g$ and $h$ replaced by $\bar{g}$ and $\bar{h}$ respectively. Then we have
\[M_p^* \leq \overline{M}_p^*, \ *= b \text{ or }s.\]

\begin{table}[htb!]
	\centering
	\begin{equation*}
		\begin{array}{c|cc|cc}
			d& w_1^b&w_2^b&w_1^s&w_2^s\\ \hline
			4&0.31782007&0.02869201&0.29587267&0.02219923\\
			5&0.25978204&0.02061192&0.26589049&0.01775655\\
			6&0.22613359&0.01567357&0.23721981&0.01412977\\
			7&0.20174586&0.01234909&0.21321446&0.01142055\\
			8&0.18261917&0.00998966&0.19333219&0.00939086\\
			9&0.16701866&0.00825099&0.17673071&0.00784594\\
			10&0.15398169&0.00693180&0.16268798&0.00664634\\
			11&0.14286290&0.00590544&0.15068272&0.00569956\\
			12&0.13327591&0.00509191&0.14029288&0.00493877\\
			13&0.12492981&0.00443650&0.13123463&0.00432016\\
			14&0.11754785&0.00389912&0.12325769&0.00380973\\
			15&0.11101571&0.00345446&0.11618447&0.00338408
		\end{array}
	\end{equation*}
	\caption{Values for $w^*$ defined in \eqref{eq:wdef}.}
	\label{tab:vector}
\end{table}

\begin{table}[htb!]
	\centering
	\begin{equation*}
		\begin{array}{c|c|c}
			d&p_c^b(d)\le b_d&p_c^s(d)\le s_d\\ \hline
			4&0.3026&0.4719\\
			5&0.2032&0.2839\\
			6&0.1620&0.2088\\
			7&0.1367&0.1676\\
			8&0.1189&0.1409\\
			9&0.1055&0.1221\\
			10&0.0950&0.1079\\
			11&0.0864&0.0969\\
			12&0.0793&0.0879\\
			13&0.0734&0.0806\\
			14&0.0682&0.0744\\
			15&0.0638&0.0691
		\end{array}
	\end{equation*}
	\caption{Upper bounds for critical probabilities.}
	\label{tab:p_cbounds}
\end{table}

Using the values in Table \ref{tab:vector} for $w^*$ defined in \eqref{eq:wdef} and the upper bounds in Table \ref{tab:p_cbounds} for $p_c^*(d)$, one can verify (by applying \eqref{eq:srcom} and \eqref{eq:rhoupperbound}) that
\[\rho(M_{b_d}^b) \leq \rho(\overline{M}_{b_d}^b)<1,\ \rho(M_{s_d}^s) \leq \rho(\overline{M}_{s_d}^s)<1.\]
Then Theorem \ref{thm:p_cbound} implies that the upper bounds in Table \ref{tab:p_cbounds} are rigorous.

\section{Proof of the main results}\label{sec:proofofmainthms}
In this section, we prove Theorems \ref{thm:per} and \ref{thm:Ising}. We first prove upper bounds for $\beta_c(d)$ and $p_c^b(d)$ which hold for all $d \geq 3$. The upper bounds for $\beta_c(d)$ follow from the infrared bounds of Fr\"{o}hlich, Simon and Spencer \cite{FSS76} directly. Even though there are very precise asymptotic results for $p_c^b(d)$ for all large $d$ \cite{Kes90,HS95,vdHS05}, we did not find any upper bounds for $p_c^b(d)$ as a simple function of $d$ which hold for all $d\geq3$ in the literature. In Lemma \ref{lem:p_cub} below, we provide such a rough bound; it has the advantage of being explicit and applicable for our proof of Theorem \ref{thm:per} for all $d\geq 13$ but loses a factor of $2$ asymptotically.

Theorem 3.1 of \cite{FSS76} gives
\begin{equation}\label{eq:beta_cub}
	\beta_c(d) \leq \frac{G_d(0)}{2d} :=\frac{1}{2d} \frac{1}{(2\pi)^d} \int_{[-\pi,\pi]^d} \frac{d k}{1-\sum_{j=1}^d \cos(k_j)/d}, \ \forall d\geq 3.
\end{equation}
\begin{lemma}\label{lem:beta_cub}
	For each $d\geq 3$, 
	\[\beta_c(d) < \frac{1}{2(d-2)}.\]
\end{lemma}
\begin{proof}
	Using $1/a=\int_0^{\infty} e^{-sa}\, ds$ for each $a>0$, we may write
	\begin{align*}
		G_d(0)&=\frac{1}{(2\pi)^d} \int_{[-\pi,\pi]^d} \int_0^{\infty} \exp[-s+s \sum_{j=1}^d \cos(k_j)/d] \, ds \, dk\\
		&=\int_0^{\infty} e^{-s} \prod_{j=1}^d \left[\frac{1}{2\pi}\int_{-\pi}^{\pi} \exp[s \cos(k_j)/d] \, d k_j\right] \, ds,
	\end{align*}
	where we have used the Fubini-Tonelli theorem in the second equality.
	Recall the modified Bessel function of the first kind of order zero
	\[I_0(z):=\frac{1}{2\pi} \int_{-\pi}^{\pi} \exp[z\cos(\theta)] \, d\theta=\frac{1}{\pi} \int_0^{\pi} \exp[z\cos(\theta)] \, d\theta.\]
	Plugging this into $G_d(0)$, we get
	\[G_d(0)= \int_0^{\infty} e^{-s} [I_0(s/d)]^d \,ds= d\int_0^{\infty}  [e^{-t}I_0(t)]^d \, dt.\]
	Theorem 3.1 of \cite{YC16} gives
	\[I_0(t)<e^t/\sqrt{1+2t}, \ \forall t>0.\]
	Therefore,
	\[G_d(0)< d\int_0^{\infty} (1+2t)^{-d/2} \, dt= d/(d-2), \ \forall d\geq 3.\]
	This combined with \eqref{eq:beta_cub} completes the proof of the lemma.
\end{proof}
We remark that the upper bound in Lemma \ref{lem:beta_cub} is slightly weaker than the inequality (12) in \cite{ABL87}.

\begin{lemma}\label{lem:p_cub}
	For each $d\geq 3$,
	\[p_c^b(d) \leq 1-\exp[-G_d(0)/d]<1/(d-2).\]
\end{lemma}
\begin{proof}
	Theorem 3.1 of \cite{Gri95} gives
	\[p_c^b(d) \leq p_c^I(d),\]
	where $p_c^I(d)$ is the critical probability for the $q=2$ random-cluster model on $\mathbb{Z}^d$. The Edwards-Sokal coupling \cite{ES88} implies
	\[p_c^I(d)=1-e^{-2\beta_c(d)}.\]
	This, \eqref{eq:beta_cub}, and Lemma \ref{lem:beta_cub} finish the proof of the lemma since $1+x \leq e^x$ for all $x\in \mathbb{R}$.
\end{proof}

\begin{proof}[Proof of Theorem \ref{thm:per}]
	Note that
	\[f(d):=(1-1/(d-2))^{2d} \text{ is increasing in } d\geq 3.\]
	By Table \ref{tab:p_cbounds}, 
	\[f(13)=(1-1/11)^{26}>0.08390>0.0806 \geq p_c^s(13).\]
	Recall the definition of $\mathcal{X}_p(d)$  in the introduction. For fixed $d$, $\mathcal{X}_p(d)$ is decreasing in $p$. Since $p^s_c(d)$ is decreasing in $d$, $f(d)\geq f(13)>p_c^s(13)\geq p_c^s(d)$ for each $d\geq 13$. Proposition~\ref{prop:per} and Lemma \ref{lem:p_cub} imply that the statement in Theorem \ref{thm:per} holds for all $d\geq 13$. For $d=9$, one may use the bounds in Table \ref{tab:p_cbounds} and Proposition \ref{prop:per} to conclude the proof
	\[(1-0.1055)^{18}>0.1344>0.1221 \geq p_c^s(9).\]
	The proof for $d=10,11,12$ is similar.
\end{proof}

\begin{proof}[Proof of Theorem \ref{thm:Ising}]
	Let
	\[f(d,\beta):=1-\frac{e^{d\beta}}{2[\cosh(\beta)]^d}=1-\frac{1}{2}\left[\frac{2}{1+e^{-2\beta}}\right]^d.\]
	For fixed $d$, $f(d,\beta)$ is decreasing in $\beta>0$. Plugging  $\beta=1/(2d-4)$ from Lemma \ref{lem:beta_cub} into $f$, we get
	\[f(d,1/(2d-4))=1-\frac{1}{2}\left[\frac{2}{1+e^{-1/(d-2)}}\right]^d \text{ is increasing in }d\geq 3.\]
	For $d=12$, we have
	\[f(12,1/20)>0.1024>0.0879 \geq p_c^s(12),\]
	where we have used Table \ref{tab:p_cbounds} in the last inequality. This and Proposition \ref{prop:Ising} prove Theorem~\ref{thm:Ising} for all $d\geq 12$. For $d=8,\dots,11$,  \eqref{eq:beta_cub} gives
	\[\beta_c(8)<0.0675, \beta_c(9)<0.0594, \beta_c(10)<0.0530, \beta_c(11)<0.0479.\]
	For $d=8$, we may use Proposition \ref{prop:Ising} and Table \ref{tab:p_cbounds} to conclude the proof
	\[f(8,0.0675)>0.1574>p_c^s(8).\]
	The proof for $d=9,10,11$ is similar.
\end{proof}

\section{Numerical estimates}\label{sec:numerical}
In this section, we first list numerical estimates for $p_c^b(d)$ and $p_c^s(d)$ with $3\leq d \leq 12$ from \cite{XWLD14} and \cite{MM18} in Table \ref{tab:p_cnumerical}. 
\begin{table}[htb!]
	\centering
	\begin{equation*}
		\begin{array}{c|c|c}
       d &  p_c^{b}(d)    &       p_c^{s}(d)\\  \hline
       3   &     0.24881185(10)   &               0.31160768(15)\\
       4   &       0.16013122(6) &             0.19688561(3)\\
       5   &        0.11817145(3)  &             0.14079633(4)\\
       6   &        0.09420165(2) &             0.109016661(8)\\
       7   &        0.078675230(2) &               0.088951121(1)\\
       8    &      0.0677084181(3)  &              0.075210128(1)\\
       9   &       0.0594960034(1) &              0.0652095348(6)\\
      10  &       0.0530925842(2)  &             0.0575929488(4)\\
      11   &       0.04794968373(8) &            0.0515896843(2)\\
      12   &      0.04372385825(10) &            0.0467309755(1)
		\end{array}
\end{equation*}
\caption{Numerical estimates of critical probabilities. The $d=3$ entries are from \cite{XWLD14}; the $d=4,\dots,12$ entries are from \cite{MM18}.}
\label{tab:p_cnumerical}
\end{table}

For the Ising model, we find the following numerical estimates for $\beta_c(d)$ from \cite{LM23} for $d=4$ and \cite{LM15} for $d=5,6,7$.
\begin{equation}\label{eq:beta_cnumerical}
	\begin{aligned}
		\beta_c(4)& \approx 0.149693785(10), &&\beta_c(5) \approx 0.11391498(2),\\
		 \beta_c(6)& \approx 0.0922982(3),  &&\beta_c(7) \approx 0.0777086(8).
	\end{aligned}
\end{equation}

\section*{Acknowledgments}
This research was partially supported by the National Natural Science Foundation of China (No. 12271284 and No. 12226001). The authors acknowledge the use of AI tools for the computations in Section \ref{sec:upperbound}.

%\section*{Data Availability}
%Data sharing is not applicable to this article as no datasets were generated or analysed during the current study.

\bibliographystyle{abbrv}
\bibliography{reference}

\end{document}